\documentclass[12pt, reqno]{amsart}
\usepackage{amsmath, amsthm, amscd, amsfonts, amssymb, graphicx, color,bm,ulem}
\usepackage[bookmarksnumbered, colorlinks, plainpages]{hyperref}
\hypersetup{colorlinks=true,linkcolor=red, anchorcolor=green, citecolor=cyan, urlcolor=red, filecolor=magenta, pdftoolbar=true}
\usepackage{setspace}
\newtheorem{theorem}{Theorem}[section]
\newtheorem{lemma}[theorem]{Lemma}
\newtheorem{proposition}[theorem]{Proposition}
\newtheorem{corollary}[theorem]{Corollary}
\theoremstyle{definition}

\theoremstyle{remark}

\numberwithin{equation}{section}

\definecolor{darkgreen}{rgb}{.1,.5,0}

\begin{document}
	
	\setcounter{page}{1}

	\title[Model Theory for Sub-$2$--normal Operators]{Model Theory for Operators Related to \\Square Roots of Normal Operators}
	
	

	
\author[R.E. Curto  and T. Prasad]{Ra\'ul E. Curto  and Thankarajan Prasad}
	
	\address{ R.E Curto, Department of Mathematics,University of Iowa, Iowa City, IA 52242, USA.}
\email{\textcolor[rgb]{0.00,0.00,0.84}{raul-curto@uiowa.edu}}
\address{T. Prasad, Department of Mathematics, Nalanda University, 
Rajgir, District Nalanda-803116, Bihar, India.}
	\email{\textcolor[rgb]{0.00,0.00,0.84}{ prasadvalapil@gmail.com} ; \textcolor[rgb]{0.00,0.00,0.84}{prasad@nalandauniv.edu.in}}
	\date{Received: date / Accepted: date}
	
	
	
	\date{Received: xxxxxx; Revised: yyyyyy; Accepted: zzzzzz.}

\begin{abstract}
	
	In this paper we prove that a square root of a cyclic normal operator is unitarily equivalent to a block multiplication operator on a vector-valued Lebesgue space with a $2$--normal symbol. \ In addition, we show that a cyclic operator which admits a cyclic $2$--normal extension is unitarily equivalent to a block multiplication operator on a corresponding  vector-valued Hardy space with $2$--normal symbol. \  We consider the existence of bounded point evaluations in the vector-valued Hardy space setting; as an application, we prove that an operator admitting a cyclic $2$--normal extension has nontrivial invariant subspaces. \ We also study uniqueness for minimal $n$--normal extensions of operators that have cyclic $2$--normal extensions.
	
\end{abstract}
\maketitle

%
%

\section{introduction}

The celebrated spectral theory of normal operators says that normal operators acting on separable Hilbert spaces are unitarily equivalent to a countable orthogonal direct sum of multiplication operators on suitable $L^{2}$ spaces. \ We recall that a Hilbert space operator is subnormal if it has a normal extension. \ The class of subnormal operators, introduced by P.R. Halmos \cite{Halmos2} and initially developed by Halmos and J. Bram \cite{Bram,Halmos1,Halmos3} is an interesting extension of the well-studied class of normal operators. \ Bram\cite{Bram} proved that subnormal operators are unitarily equivalent to multiplication operators on a Hardy space and, as a result, multiplication operators on $H^{2}$ spaces can be regarded as the universal model for subnormal operators. \ This result is the  catalyst  for the study of the connections between subnormal operator theory and analytic function theory. \ 

Let $\mathcal{H}$ and $\mathcal{K}$ be separable complex Hilbert spaces, and let $B(\mathcal{H},\mathcal{K})$ denote the algebra of all  bounded linear operators from $\mathcal{H}$ to $\mathcal{K}$; we also write $B(\mathcal{H})=B(\mathcal{H},\mathcal{H})$. \ Recall that an operator $T\in B(\mathcal{H})$ is said to be \textit{$n$--normal} if  $T^{*} T^{n} = T^{n}T^{*}$\cite{patel}. \ Alternatively, an operator $T$ is $n$--normal if and only if $T^{n}$ is normal.

The present authors initiated the study of some natural generalizations of subnormal operators, namely, the class of sub-$n$--normal operators \cite{CP}. \ An operator $T\in B(\mathcal{H})$ is said to be $n$--subnormal if $T^{n}$ is subnormal. \ An operator $T\in B(\mathcal{H})$ is said to be sub-$n$--normal if it is the restriction of an $n$--normal operator to an invariant subspace; that is, there exists a Hilbert space $\mathcal{K}$ containing $\mathcal{H}$ and an $n$--normal operator $S$ on $\mathcal{K}$ such that $\mathcal{H} \subseteq \mathcal{K}$ and $Tx=Sx$ for all $x\in \mathcal{H}$ \cite{CP}.

Our interest in theses classes of operators is partly motivated by a long-standing open question in operator theory, recorded as Problem 5.6 in \cite{ConwayFeldman}: Characterize the subnormal operators having a square root.\\

The above-mentioned classes can be summarized in the following diagram from \cite{CP}:
 
$\begin{array}{ccccc} 
&& \framebox[1.7in][c]{
$\begin{array}{c}
\textrm{subnormal} \\
\end{array}$
} && \\
& \mbox{\huge{\rotatebox{45}{$\Longleftarrow$}}} \hspace*{-.1in} && \mbox{\huge{\rotatebox{90}{\rotatebox{45}{$\Longleftarrow$}}}} \\

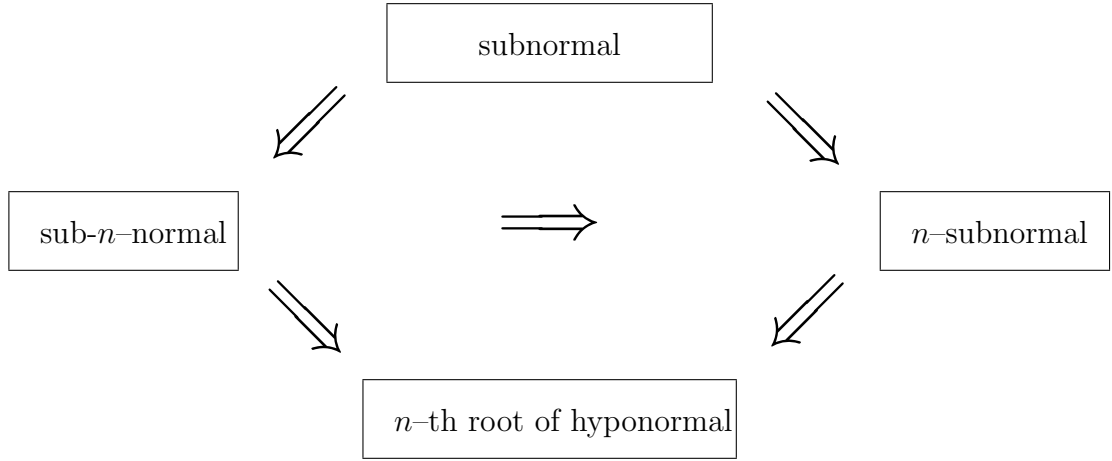
\framebox[1.2in][l]{
$\begin{array}{c}
\textrm{sub-$n$--normal} \\
\end{array}$ 
}  &&\mbox{\huge{$\Longrightarrow$}}&& 
\framebox[1.2in][l]{
$\begin{array}{c}
\textrm{$n$--subnormal} \\
\end{array}$
} \\
&  \mbox{\huge{\rotatebox{90}{\rotatebox{45}{$\Longleftarrow$}}}} \hspace*{-.1in} && \mbox{\huge{\rotatebox{45}{$\Longleftarrow$}}} & \label{diagram}  \\
&& \framebox[1.95in][l]{
$\begin{array}{c}
\textrm{$n$--th root of hyponormal} \\
\end{array}$}
\end{array}$

\bigskip

Consider now the Hilbert space $\ell^2$ with its standard orthonormal basis $\{e_j\}_{j=0}^\infty$ (note that we begin indexing at zero). \ Given a bounded sequence of positive real numbers $\alpha \equiv \{\alpha_j\}_{j \ge 0}$, we define the {\it unilateral weighted shift} $W_\alpha$ acting on $\ell^2$ by $W_\alpha e_j := \alpha_j e_{j+1}$, and extend it to all of $\ell^2$ by linearity. \  It is well-known that (i) $W_{\alpha}$ is never normal; (ii) quasinormal if and only if it is a scalar multiple of the (unweighted) unilateral shift $U_+$; and (iii) hyponormal if and only if the sequence $\alpha$ is non-decreasing. \ 

On the other hand, recall that the Hardy space of the unit circle $\mathbb{T}$ is the closed subspace $H^2 \equiv H^2(\mathbb{T})$ of $L^2 \equiv L^2(\mathbb{T},\frac{d\theta}{2 \pi})$ spanned by the polynomials $\mathbb{C}[z]$. \ The above-mentioned unilateral shift $U_+$ is (canonically) unitarily equivalent to the multiplication operator $M_z \in B(L^2(\mathbb{T},\frac{d \theta}{2 \pi}))$ restricted to $H^2(\mathbb{T})$. \ 

As is customary, we let $M_n \equiv M_{n \times n}$ denote the algebra of $n \times n$ matrices over $\mathbb{C}$. \ We denote by $L^2 _{\mathbb{C}^n}$ (resp. $L^{\infty} _{M_n}$) the Hilbert space of all $\mathbb{C}^n$--valued Lebesgue square integrable functions on the unit circle (resp. the Banach space of all $M_n$--valued essentially bounded functions on the unit circle).  

%
%

\section{Model theory}

We begin with a notion introduced in \cite{CP}. \ Given the unit circle $\mathbb{T} \subseteq \mathbb{C}$ and the $C^*$--algebra $M_2 \equiv M_2(\mathbb{C})$ of $2 \times 2$ matrices over $\mathbb{C}$, consider the class $\mathcal{N}_2^{(n)}$ of functions $\Phi:\mathbb{T} \rightarrow M_2$ such that $\Phi(z)^n$ is a normal $2 \times 2$ matrix a.e. on $\mathbb{T}$. \ For $\Phi \in L^{\infty}_{M_{2}} \cap \mathcal{N}_2^{(n)}$, it was shown in \cite[Example 3.10]{CP} that the block multiplication operator $M_\Phi$ is $n$--normal. \ We will now establish a converse, that is, we will show that that every cyclic $2$--normal operator is (unitarily equivalent to) a multiplication operator $M_{\Phi} \in B(L^{2}(,\mu)_{\mathbb{C}^2})$, whose symbol $\Phi$ is a bounded $2$--normal function.  

We first need two well-known elementary lemmas, whose proofs we include for the sake of completeness.

\begin{lemma} \label{elemlem}
Let
$$
R = \begin{bmatrix}
      D& E\\
     0& F\\
 \end{bmatrix}
$$
be an arbitrary $2 \times 2$ operator matrix acting on $\mathcal{H} \oplus \mathcal{H}$, and assume that $R$ has a cyclic vector $(x,y)$, where $x,y \in \mathcal{H}$. \ Then $F$ is cyclic, with cyclic vector $y$.
\end{lemma}

\begin{proof}
Observe that, for every integer $n \ge 2$, we have
$$
R^n = \begin{bmatrix}
      D^n & D^{n-1}E+D^{n-2}EF+ \cdots + DEF^{n-2}+EF^{n-1}\\
     0& F^n\\
 \end{bmatrix}.
$$
It readily follows that, given a polynomial $p \in \mathbb{C}[z]$, we have
$$
p(R)\begin{bmatrix}
      x\\
     y\\
 \end{bmatrix} = \begin{bmatrix}
      p(D)x + z\\
     p(F)y\\
 \end{bmatrix},
$$
for some vector $z \in \mathcal{H}$. \ Since every vector $(0,z)$ can be approximated by a sequence of vectors $\{p_k(R)(x,y)\}_{k \ge 1}$, we conclude that $z$ can be approximated by a sequence of vectors of the form $\{p_k(F)y\}_{k \ge 1}$, and therefore $F$ is cyclic with cyclic vector $y$.
\end{proof}

\begin{lemma} \label{newlem}
Let $S$ and $T$ be $2$--normal operators such that $ST+TS=0$. \ Then $S+T$ is $2$--normal.
\end{lemma}

\begin{proof}
Observe that 
$$
(S+T)^2=S^2+ST+TS+T^2=S^2+T^2.
$$
Since $S^2$ and $T^2$ are normal, it readily follows that $S+T$ is also $2$--normal.
\end{proof}
 
\begin{theorem} \label{thmcyclic} 
Let $T\in B(\mathcal{H})$ be a cyclic $2$--normal operator. \ Then there exists a finite measure space $(X,\mu)$ and a  $2$--normal symbol $\Psi \in$ $L^{\infty}_{M_2}$ such that $T$ is unitarily equivalent to $M_{\Psi} \in B(L^{2}(X,\mu)_{\mathbb{C}^2})$. 
\end{theorem}

\begin{proof}
Since $T$ is $2$--normal operator, we can apply \cite[Theorem 1]{RadRos} to write 
$$
T=A \oplus \begin{bmatrix}
      B& C\\
     0& -B\\
 \end{bmatrix},
$$
where $A$ and $B$ are normal and $C$ is a positive one-to-one operator commuting with $B$. \ Since $T$ is cyclic, without loss of generality we may assume that $T$ has no normal part, that is,   
$$
T=\begin{bmatrix}
      B& C\\
     0& -B\\
 \end{bmatrix}.
$$
By Lemma \ref{elemlem}, we know that $B$ cyclic. \ Since $B$ is also normal, there exists a finite measure space $(X,\mu)$ and a function $\phi \in L^{\infty}(X,\mu)$ such that $B$ is unitarily equivalent to the operator $ M_{\phi}$ acting  on $L^{2}(X,\mu)$. \ (In view of the cyclicity of $T$, we may even assume that $\phi(z)=z$ for all $z \in X$ (see \cite[IX.8]{Conway1}). \ Thus, $T$ can be written as
$$
\begin{bmatrix}
      UM_{\phi}U^{*}& C\\
     0&  -UM_{\phi}U^{*}
 \end{bmatrix} = 
\begin{bmatrix}
      UM_{\phi}U^{*}& 0\\
     0&  -UM_{\phi}U^{*}
      \end{bmatrix} +  \begin{bmatrix}
      0& C\\
     0& 0
 \end{bmatrix}.
$$
That is, 
$$
T = \begin{bmatrix}
      U& 0\\
     0&  U
      \end{bmatrix} \begin{bmatrix}
      M_{\phi}& 0\\
     0&  -M_{\phi}
      \end{bmatrix} \begin{bmatrix}
      U^{*}& 0\\
     0&  U^{*}
      \end{bmatrix} + \begin{bmatrix}
      U& 0\\
     0&  U
      \end{bmatrix}   \begin{bmatrix}
      U^{*}& 0\\
     0&  U^{*}
      \end{bmatrix}  \begin{bmatrix}
      0& C\\
     0& 0
 \end{bmatrix}    \begin{bmatrix}
      U& 0\\
     0&  U
      \end{bmatrix}  \begin{bmatrix}
      U^{*}& 0\\
     0&  U^{*}
      \end{bmatrix}.
			$$
Then 
$$
T = \begin{bmatrix}
      U& 0\\
     0&  U
      \end{bmatrix}    \Bigg(     \begin{bmatrix}
      M_{\phi}& 0\\
     0&  -M_{\phi}
      \end{bmatrix}   +  \begin{bmatrix}
      0&U CU^*\\
     0& 0
 \end{bmatrix}     \Bigg) \begin{bmatrix}
      U^{*}& 0\\
     0&  U^{*}
      \end{bmatrix} ,$$   
where $ \begin{bmatrix}
      U& 0\\
     0&  U
      \end{bmatrix} $ is a unitary operator on $L^{2}(X,\mu)_{\mathbb{C}^2}$  to $\mathcal{H}$.
      
Since $T$ is cyclic, it follows that $B$ is cyclic, by Lemma \ref{elemlem}. \ Since $CB=BC$ and since $B$ is cyclic and also unitarily equivalent to the multiplication operator $ M_{\phi}$ on $L^{2}(X,\mu)$, it follows that $C$ is a multiplication operator $M_{\xi}$ on $L^{2}(X,\mu)$.

That is, $T$ is unitarily equivalent to 
$$
M_{\Phi} + M_{\Xi} = M_{\Phi + \Xi} \in B(L^{2}(X,\mu)_{\mathbb{C}^2})
$$
where  $\Phi(z):=\begin{bmatrix}
      z& 0\\
     0&  -z
      \end{bmatrix}$ and $\Xi(z):=\begin{bmatrix}
      0& z\\
     0&  0\
      \end{bmatrix}$.
      
  Since 
      $\Phi(z)$ and $\Xi$ are $2$--normal, Lemma \ref{newlem} implies that $\Psi:=\Phi +\Xi$ \; is $2$--normal. \ Thus, $T$ is unitarily equivalent to a multiplication operator with $2$--normal symbol.
      
Conversely, suppose that $T$ is unitarily equivalent to $M_{\Phi} +M_{\Xi}$, where 
$\Phi(z)=\begin{bmatrix}
      z& 0\\
     0&  -z
      \end{bmatrix}$.
				
				It is easy to see that  $M_{\Phi} M_{\Xi} + M_{\Xi} M_{\Phi} =0$, which implies that $M_{\phi} C=C M_{\phi}$. \ Using this equality, we see that $ ( M_{\Phi} +M_{\Xi})^{2}$ is normal; that is, $ M_{\Phi} + 
				M_{\Xi}$ is $2$--normal.
      \end{proof}

Given a positive integer $k$, let $N_{n}$ denote the set of $n$--normal functions. \ If $\Phi \in L^{\infty}(X,\mu)  \cap  N_{n}$, then $M_\Phi$, acting on  $L^2(X,\mu)_{\mathbb{C}^n }$, is $n$--normal. \ It follows that, for $\Phi  \in H^{\infty}(X,\mu) \cap  N_{n}$, $M_\Phi$ restricted to the invariant subspace $H^2(X,\mu)_{\mathbb{C}^n}$ is sub-$n$--normal \cite{CP}. \ (Here $H^{2}(X,\mu)_{\mathbb{C}^{2}}$ denotes the closure of $\mathbb{C}[z]_{\mathbb{C}^2}$ in $L^{2}(X,\mu)_{\mathbb{C}^{2}}$.) \ The following theorem says that $M_\Phi$ is unitarily equivalent to a sub-$2$--normal operator, with the symbol $\phi$ both analytic and $2$--normal.

\begin{theorem} \label{t2}
Let $T$ be a cyclic operator with a cyclic $2$--normal extension. \ Then $T$ is unitarily equivalent to a block multiplication operator $M_\Psi$ acting on a suitable Hardy space $H^{2}(X,\mu)_{\mathbb{C}^{2}}$, where
$$
\Psi(z) := \begin{bmatrix}
     z& z \\
     0& -z
     \end{bmatrix} \quad (z \in \mathbb{C}).
		$$
		
\end{theorem}

\begin{proof}
Let $S$ be the cyclic $2$--normal extension of the cyclic operator $T$. \ Then $S$ is unitarily equivalent to a block multiplication operator $M_{\Psi}$ on $L^2_{\mathbb{C}^2} \equiv L^2(X,\mu)_{\mathbb{C}^2}$. \ That is, $S=R  M_{\Psi} R^{*}$, where $R=\begin{bmatrix}
      U & 0\\
     0 & U
      \end{bmatrix}$ is an isometric isomorphism from $L^{2}_{\mathbb{C}^2}$ to $\mathcal{H} \oplus \mathcal{H}$.
			
Now
\begin{align*}
|| p(T) x || ^{2}&=\langle p(T)x,  ~  p(T) x\rangle \\&= \langle p(S)x, ~(p(S))x\rangle _{\mathcal{H}}\\\ &
= \langle R ~ p(M_{\Psi})  R^{*}x, ~ R ~ p( M_{\Psi})     )  R^{*}x\rangle _{L^2_{\mathbb{C}^2} }\\&
=\langle    M_{p (\Psi)}  R^{*}x,   M_{p(\Psi)} R^{*}x\rangle _{L^2_{\mathbb{C}^2} }\\&
=\int_{X}  |p(\Psi)(R^{*}x)|^2  d \mu=\int_{X}  |p(\Psi)(y))|^{2} d \mu
\end{align*}
where $y:=R^{*}x \in L^{2}(X,\mu)_{\mathbb{C}^2}$.
 
Define $W(p(\Psi))(y):= p(T)x$. \ We see that  $ W$ is an isometry. \ Since range of $W$ is dense in $\mathcal{H}$,  we have that $W$ is an isometric isomorphism from  $H^{2}(\mathbb{D},\mathbb{C}^2)$ onto $\mathcal{H} \oplus \mathcal{H}$.

Let $q(z):=zp(z) \; (z \in \mathbb{C})$, and observe that $Tp(T^{})x = q(T^{})x$. 
     
Now
\begin{align*}
W^{-1}T W p(\Psi)(z) (y)& = W^{-1}T p(T) x \\  & = W^{-1}q(T) x\\ &= q(\Psi(z))y\\&  
=\begin{bmatrix} 
     z& z\\
     0& - z
     \end{bmatrix}  p\left( \begin{bmatrix} 
     z& z\\
     0&  -z\
     \end{bmatrix} \right) (y)
\\ &=   
M_{    \begin{bmatrix} 
     z& z\\
     0&  -z
     \end{bmatrix} } p(\Psi(z))(y)
\end{align*}
Therefore, $T$ is unitarily equivalent to  $M_{\Psi(z)}$.
\end{proof}

\subsection{An Application to Bounded Point Evaluations} \ Given a positive Borel measure $\mu$ supported by a compact subset $K$ of $\mathbb{C}$, let $P_{C^{n}}^{t}(\mu)$ be the closure of $C_{n}(z)$ in $L_{C^{n}}^{t}(\mu)$, where
$$
C_{n}(z):=\{ P(z): P=(p_{1}, \cdots p_{n}): p_{i} \textrm{ a polynomial in a complex variable} z ; (i=1,\ldots,n)\}.
$$
A point $\lambda \in  \mathbb{C}$ is a bounded point evaluation for $P_{C^{n}}^{t}(\mu)$ if $F: C_{n}(z) \rightarrow \mathbb{C}^n$, defined by
$$
F(P):=P(\lambda)  (= (p_{1}(\lambda) \cdots p_{n}(\lambda)),
$$
has a bounded extension to  $P_{C^{n}}^{t}(\mu)$, also denoted by $f$. \ That is, we must have
$$
|| f(P)|| \leq M ||P||_{ P_{C^{n}}^{t}}
$$
for all $P \in P_{C^{n}}^{t}(\mu)$. \ (For additional information on bounded point evaluations, we refer the reader to \cite{Bre}, \cite{Conway2}, \cite{ConwayFeldman} and \cite{ConYang}.) 

Bounded point evaluations (BPE) exist for $P_{C^{n}}^{t}(\mu)$ if and only if the block multiplication operator $M_\lambda$ has a nontrivial invariant subspace. \ For, if $z_{0}\in \mathbb{C}$, $M_{z_{0}}=\{ f \in  P_{C^{n}}^{t}(\mu):  f(z_{0})=0\}$ is a nontrivial invariant under multiplication $M_{\Psi}$  on $L_{C^{n}}^{t}(\mu)$. \ If  $P_{C^{n}}^{2}(\mu) \neq L_{C^{n}}^{2}(\mu)$ and $\lambda$ is a bounded point evaluation for $P_{C^{n}}^{2}(\mu)$, then the set of  functions in $P_{C^{n}}^{2}(\mu)$ which vanish on $\lambda$  is a nontrivial invariant subspace for $M_{\Psi}$.


As a result, to settle in the affirmative the invariant subspace problem for a class of operators unitarily equivalent to the block multiplication operator $M_{\phi}$ on $H^2_{\mathbb{C}^n}$, it is sufficient to prove that $H_{C^{n}}^{2}(\mu) $ has a BPE. 

By Thomson's theorem \cite{Tho}, the classical Hardy spaces have BPE and so does $H_{C^{n}}^{2}(\mu)$. \ Then, by Theorem \ref{t2}, if $T$ is a cyclic operator with a cyclic $2$--normal extension, then $T$ has a nontrivial invariant subspace.

%
%

\section{sub-$2$--normal unilateral weighted shifts}

We now extend, to  the case of sub-$2$--normal unilateral weighted shifts, some classical results obtained, for subnormal operators, by C. Cowen and J. Long \cite{cl} and by R. Gellar and L.J. Wallen \cite{GW}. \ First, we need some notation. \ Let $\{e_{i}\}_{i\ge 0}$ be an orthonormal basis for $\mathcal{H}$, let $\alpha \equiv \{\alpha_i\}_{i \ge 0}$ be a sequence of positive numbers, and let $W_\alpha$ be the associated unilateral weighted shift, given by $W_\alpha e_{i}:=\alpha_{i}e_{i+1} \; (i \ge 0)$.  

\begin{theorem} \label{thmc}
Let $W_\alpha$ be a unilateral weighted shift with weight sequence $\alpha$, and assume that $\sup_i \alpha_{i}=1$. \ Then $W_\alpha$ is a cyclic operator that admits a cyclic $2$--normal extension if and only if there is a measure $\mu$  on $[0,1]$ such that
	$$(\alpha_{0}\alpha_{1}\cdots\alpha_{2i-1} )^{2}=\int_{0}^{1}r^{4i}d\mu(r).$$
\end{theorem}
  
\begin{proof}
From Theorem \ref{t2}, there is a unitary operator $U: \mathcal{H}\rightarrow  H^{2}(\overline{\mathbb{D}}, \mathbb{C}^{2}) $  such that    $T$ is unitarily equivalent to a block multiplication operator $ M_{\Psi} $ on $H^{2}(\overline{\mathbb{D}}, \mathbb{C}^{2})$, where $\Psi(z) =\begin{bmatrix} 
     z& z\\
     0& -z
     \end{bmatrix} \quad (z\in \mathbb{C})$ and $H^{2}(\overline{\mathbb{D}}, \mathbb{C}^{2})$ denotes the closure of $\mathbb{C}[z]_{\mathbb{C}^2}$ in $L^2(X,\mu)_{\mathbb{C}^2}$. \ We readily see that $||Ue_{0}||=1$. \ 
		
		If we define $\beta_0:=1$ and $\beta_{2i}:=\alpha_{0}\alpha_{1}\cdots\alpha_{2i-1} \; (i\geq 1)$, then $T^{2i}e_{0}=\beta_{2i}e_{2i}$. \\
Therefore,
$$
UT^{2i}e_{0}=M_{\Psi} Ue_{0}= \begin{bmatrix} 
     z& z\\
     0& -z
     \end{bmatrix}^{2i} Ue_{0}.
$$
If we  let $ r^{2i}:=||\begin{bmatrix} 
     z& z\\
     0& -z
     \end{bmatrix}^{2i} Ue_{0}||^2,$  then\\                    
  \begin{align*}                        
    \int |r|^{4i}d\mu = \int |r^{2i}|^{2}d\nu   = ||UT^{2i}e_{0}||^{2} = ||\beta_{2i}e_{ni}||^{2} =\beta^{2}_{2i}.
  \end{align*}
If $\nu$ is a measure defined on $[0,1]$ by $\nu(\Delta):=\mu(\{r: |r^{2i} |\in \Delta \}) \; (\textrm{for } \Delta \textrm{ a Borel subset of } [0,1])$, then the desired result follows.
\end{proof}

To simplify the notation, in the next result we allow the weight sequence $\alpha$ to be indexed by $\mathbb{N}$ instead of $\mathbb{N} \cup \{0\}$.
  
\begin{theorem}\label{cl}
\ For $n$ fixed and $0<r<1$, the weighted shift $W_\alpha$ with weights $\alpha_{i}=(1-r^{4i(2i+1)})^{\frac{1}{2}} \; \; (i=1,2,\ldots)$ is a cyclic operator with a cyclic $2$--normal extension.
\end{theorem}
\begin{proof} \ Recall first the 
	$q$--binomial formula \cite[p.~350]{As},
	$$
	1+ \sum_{k=1}^{\infty} \frac{(a;q)_{k}}{(q;q)_{k}}x^{k}= \frac{(ax;q)_{\infty}}{(x;q)_{\infty}},
	$$
	where $(a;q)_{k}=(1-a)(1-aq)....(1-1q^{k-1})$. 
	
	If we now let $a:=0$, $x:= r^{4(2i+1)}$ and $q:=r^{4}$, it follows that
	\begin{align*}
		1+ \sum_{k=1}^{\infty}\frac{ (r^{4(2i+1)})^{k}}{(1-r^{4})(1-r^{8})\cdots (1-r^{4k})}&=\prod_{j=ni+1}^{\infty}(1-r^{4j})^{-1}\\
		&=\pi_{\infty}^{-1}(1-r^{4})(1-r^{8})\cdots (1-r^{4(2i)}),
	\end{align*}
	where $\pi_{\infty}:=\prod_{j=1}^{\infty}(1-r^{4j})$.
	Thus
	\begin{align*}(\alpha_{0}\alpha_{1}\cdots\alpha_{2i-1} )^{2}
		&=(1-r^{4})(1-r^{8})\cdots (1-r^{4(2i)})\\
		&=\bigg[1+\sum_{k=1}^{\infty}r^{4k}(1-r^{4})^{-1}(1-r^{8})^{-1}\cdots (1-r^{4k})^{-1}(r^{k})^{4(2i)}\bigg]\pi_{\infty}\\
		&=\int_{0}^{1}t^{4i}d\mu(t).
	\end{align*}
	where $\mu$ is the discrete measure on $[0,1]$ with mass $\pi_{\infty}$ at 1 and $\pi_{\infty}r^{4k}(1-r^{2n})^{-1}(1-r^{8})^{-1}\cdots (1-r^{4k})^{-1}$ at $r^{k}$. \ Then by Theorem \ref{thmc}, it follows that the weighted shift $W_\alpha$ with weights $\alpha_{n}=(1-r^{4(2i+1)})^{\frac{1}{2}}$ for fixed $n$ and $i=1,2,.....$ is a cyclic operator that has a cyclic $2$--normal extension.
\end{proof}

\begin{theorem}	On a Hilbert space $\mathcal{H}$, let $T$  be a  cyclic operator admitting a cyclic $2$--normal extension, and let $x$ be a nonzero vector in $\mathcal{H}$. \ Then the weighted shift $W$ on $\ell^{2}$ with weight sequence $\{||T^{ni}x|| / ||T^{ni-n}x||  \}^{\infty}_{i=1}$ is a cyclic operator that admits a cyclic $2$--normal extension.
\end{theorem}

\begin{proof}
Since $T$ admits a cyclic $2$--normal extension $N$, there is a positive operator-valued measure $F$ supported on some interval $[0,a]$ in $\mathbb{R}$ such that $N^{*2i}N^{2i}=\int t^{4i}dF(t)$, by \cite[Theorem 3.6]{CP}. \ Then, for $0\neq x \in \mathcal{H}$,  $||T^{2i}x||^{2}=\int t^{4i}dQ(t)$, where $dQ(t)=\langle dF(t)x,x \rangle$. \ For a unit vector x, the $i$--th partial product of the weight sequence is  $||T^{2i}x||$. \ Therefore, by Theorem \ref{thmc}, $W \equiv T_{x}$ is a cyclic operator admitting a cyclic $2$--normal extension for every unit vector $x \in \mathcal{H}$, and therefore for every nonzero vector $x\in \mathcal{H}$.
\end{proof}

%
%

\section{Minimal $n$--normal Extensions}

Let $T\in B(\mathcal{H})$  be a sub-$n$--normal operator and let $S \in B(\mathcal{H})$ be an $n$--normal operator. \ Since $T$ is sub-$n$--normal, we
can find an $n$--normal extension  $\tilde{S}$ on a Hilbert space  $\mathcal{K}$, where $\mathcal{H} \subseteq \mathcal{K}$ and $\tilde{S} |_{\mathcal{H}} = T$. \ In light of \cite[Theorem 2.18]{patel}, we  see that $\tilde{S} \oplus S $ is an  $n$--normal operator on $\mathcal{K} \oplus  \mathcal{H}$, which is an extension for $T\oplus S$. \ Therefore, $T\oplus S$ is sub-$n$--normal. \ Also, $S' \oplus S$  can be identified as an $n$--normal extension of $T$. \ That is, $n$--normal extensions of a sub-$n$--normal operator are not necessarily unique (up to unitary equivalence). \ The following result gives a criterion for minimality.

\begin{theorem} \label{thm 229}
	Let $T$  be an operator that admits an extension $S$ which is $m$--normal for every $m\geq n$. \ Then $S$ is a minimal $n$--normal extension of $T$ if and only if 
	$$
	\mathcal{K}=
	\bigvee\{(S^{*m})^{k}h; h\in \mathcal{H} ,  k\geq 0 , \text{ and
	} m \geq n\}.
	$$
\end{theorem}	

\begin{proof}

Write 
$$
\mathcal{L}:=\bigvee\{(S^{*m})^{k}h; h\in \mathcal{H} ,  k\geq 0, \textrm{ and } m\geq n\}.
$$
It is evident that $\mathcal{H}\subseteq \mathcal{L}$. \ Since  $S$ is $m$-normal for all $m \geq n$, it follows that $S^{*}$ is $m$--normal for all $m \geq n$. \ Then $ S ( S^{*km}h)= S^{*km}(Sh)= S^{*km}(Th) $ (using the fact that $Sh=Th$ for $h\in \mathcal{H}$). \ 	Thus, $ S ( S^{*km}h) =S^{*km} h_{1}$, where $h_1 := Th \in \mathcal{H}$.	\ That is, $S\mathcal{L}\subseteq \mathcal{L}$. \ Also, $S^* ( S^{*km}h)=S^{*(km+1)}h \in \mathcal{L}$, for all $h\in \mathcal H, k \ge 0, m \ge n$. \ Thus, $S^* \mathcal L \subseteq \mathcal L$. \ Let $P_{\mathcal L}$ denote the orthogonal projection of $\mathcal H$ onto $\mathcal L$, and let $S_{\mathcal L}$ denote the compression of $S$ to $\mathcal L$. \ Since $S^{*n}f \in \mathcal{L}$ and $S \mathcal{L}\subseteq \mathcal{L}$, we have 
	$$
	S_{\mathcal{L}}^{*n}S_{\mathcal{L}}f=P_{L}S^{*n}Sf=P_{\mathcal{L}}SS^{*n}f=SS^{*n}f
	$$
	and 	
	$$
	S_{\mathcal{L}}S_{\mathcal{L}}^{*n}f=SP_{\mathcal{L}}S^{*n}f=SS^{*n}f,
	$$
	for all $f\in \mathcal{H}$. \ Therefore, $S_{\mathcal{L}}$ is $n$--normal and $\mathcal{H}\subseteq ker[S_{\mathcal{L}}^{*n}, S_{\mathcal{\mathcal{L}}}]$. \ We need to show that $\mathcal{L}$ is minimal. \ 
	
\noindent Let $\mathcal{M}\subseteq \mathcal{K}$ be such that $S\mathcal{M}\subseteq \mathcal{M}$ and $S|_{\mathcal{M}}$ is an $n$--normal extension of $T$. \ Since $\mathcal{H}\subseteq ker[S_{\mathcal{M}}^{*n}, S_{\mathcal{M}}]$, it follows that $P_{\mathcal{M}}S^{*n}Sf= SP_{\mathcal{M}}S^{*n}f$ for $f\in \mathcal{H}$.\\
	Then
	\begin{align*}
		\langle S^{*}f,  P_{\mathcal{M}}S^{*n}f\rangle&= 	\langle f,  SP_{\mathcal{M}}S^{*n}f\rangle\\
		&=\langle f,  P_{\mathcal{M}}S^{*n}Sf\rangle\\
		&=\langle f, S^{*n}Sf\rangle\\
		&=\langle f, SS^{*n}f\rangle\\
		&=	\langle S^{*}f, S^{*n}f\rangle
	\end{align*}
	Hence $S^{*n}f=P_{\mathcal{M}}S^{*n}f$ and so $S^{*n}f \in \mathcal{M}$, i.e., $\mathcal{L}\subseteq \mathcal{M}$. \ This completes the proof.
\end{proof}

\begin{proposition}
	For $k=1,2$, let $T_{k} \in \mathcal{B}(\mathcal H_k)$ be  an operator that admits an extension that is $m$--normal for every $m\geq n$, and let $S_{k}$ be a minimal $n$--normal extension, acting on the Hilbert space $\mathcal{K}_{k} \supseteq \mathcal{H}_k$. \ If $T_{1}$  and $T_{2}$ are unitarily equivalent then so are the minimal $n$--normal extensions $S_1$ and $S_2$. 
\end{proposition}

\begin{proof}
Let $\mathcal{K}_1$ and $\mathcal{K}_2$ be as in Theorem \ref{thm 229}, and let $\mathcal{L}_1$ and $\mathcal{L}_2$ be defined as in the Proof of Theorem \ref{thm 229}. \ Suppose that $U: \mathcal{L}_{1} \rightarrow \mathcal{L}_{2}$ is a unitary operator such that $UT_{1}=T_{2}U$. \ Define $V$ on $\mathcal{K}_{1}$ by
	$$
	V(S^{*nk}_{1}h):=S^{*nk}_{2}Uh, ( h\in \mathcal{H}_{1}, k \ge 0).
	$$
Note that $V\mid_{\mathcal{H}_{1}}=U$. \ For $f,g \in \mathcal{L}_{1}$,
		\begin{align*}
			||Uf+S^{*n}_{2}Ug||& =( Uf+S^{*n}_{2}g, Uf+S^{*n}_{2}g)\\
			&=(Uf,Uf)+(Ug, S^{n}_{2}Uf)+ (S^{n}_{2}Uf, Ug)+ (S^{n}_{2}Ug, S^{n}_{2}Ug) \\ &(\textrm{because } S^{n}_{2}S^{*n}_{2}=S^{*n}_{2}S^{n}_{2} \textrm{ for all } h \in \mathcal{L}_{2}) \\
			&=(Uf,Uf)+(Ug, US^{n}_{1}f)+ (US^{n}_{1}f, Ug)+ (US^{n}_{1}g, US^{n}_{1}g) =||Uf+S^{*n}_{1}g || \\
			&(\textrm{because U is unitary  and } UT_{1}=T_{2}U \Rightarrow US_{1}=S_{2}U \Rightarrow US^{n}_{1}=S^{n}_{1}U ).
		\end{align*}
Thus,
		$$
		V (f+S^{*n}_{1}g)=Uf+S^{*n}_{2}Ug.
		$$
As a result, $V$ extends to a unitary operator from $\mathcal{K}_{1}$ onto $\mathcal{K}_{2}$. \ Moreover,  $VS_{1}=S_{2}V $ holds from the following observation: given $h\in \mathcal{H}_{1}$, we have
		$$
		VS_{1}h=US_{1}h=S_{2}Uh=S_{2}Vh
		$$
and
		$$
		VS_{1}(S^{*n}_{1}h)=VS^{*n}_{1}S_{1}h=S^{*n}_{2}US_{1}h=S^{*n}_{2}S_{2}Uh=S_{2}S^{*n}_{2}Uh=S_{2}V(S^{*n}_{1}h).
		$$
		This completes the proof.
	\end{proof}
	
	\begin{corollary}
		Let $T\in \mathcal{B}(\mathcal{H})$ be an operator that admits an extension that is $m$--normal for every $m\geq n$, and let $S_{1}$  and $S_{2}$ be minimal $n$--normal extensions of $T$. \ Then  $S_{1}$  and $S_{2}$ are  unitarily equivalent.
	\end{corollary}

\noindent\textit{Acknowledgments}. \ R.E. Curto is partially supported
by U.S. National Science Foundation grant DMS-2247167. \ T. Prasad is
supported in part by the Mathematical Research Impact Centric
Support, MATRICS (MTR/2021/000373) by SERB, Department of Science
and Technology (DST), Government of India.

%
%

\end{document}